\documentclass[a4paper]{amsart}
\usepackage{amsmath,amssymb,amsfonts}
\usepackage{mathrsfs,latexsym,amsthm,enumerate}
\usepackage{amscd}
\usepackage[all]{xy}

\newtheorem{theorem}{Theorem}[section]
\newtheorem{lemma}[theorem]{Lemma} 
\newtheorem{corollary}[theorem]{Corollary}

\newtheorem{proposition}[theorem]{Proposition}
\newtheorem{remark}[theorem]{Remark}

\title[A non-commutative Stone duality]{Boolean sets, skew Boolean algebras and a non-commutative Stone duality}

\author{Ganna Kudryavtseva}
\address{G.~Kudryavtseva, University of Ljubljana,
Faculty of Computer and Information Science,
Tr\v{z}a\v{s}ka cesta 25,
SI-1001, Ljubljana,
SLOVENIA}
\email{ganna.kudryavtseva\symbol{64}fri.uni-lj.si}
\author{Mark V. Lawson}
\address{M.~V.~Lawson, Department of Mathematics
and the
Maxwell Institute for Mathematical Sciences, 
Heriot-Watt University,
Riccarton,
Edinburgh~EH14~4AS,
SCOTLAND}
\email{markl@ma.hw.ac.uk}

\thanks{The first author was partially supported by an ARRS grant P1-0288, and the second by an EPSRC grant EP/I033203/1}

\begin{document}
\maketitle
\begin{abstract}
We describe right-hand skew Boolean algebras in terms of a class of presheaves of sets over Boolean algebras called Boolean sets, 
and prove a duality theorem between Boolean sets and \'etale spaces over Boolean spaces.
\end{abstract}

\section{Introduction}

This paper is part of an ongoing collaboration exploring the connections between our different
generalizations \cite{K,Law,Law1} of classical Stone duality.
We first find an alternative description of the skew Boolean algebras which are the focus of the first author's generalization
and then reprove the main duality theorems of \cite{K} in these new terms.
These results will form the basis of \cite{KL4}, where we shall show explicitly how our two generalizations fit together
to provide a single duality theorem.
In this introduction, we define the structures we shall be studying and state the two theorems we shall be proving.

\subsection{Classical  Stone duality}

Throughout this paper, we shall use the term {\em Boolean algebra} to mean what is usually called a  {\em generalized Boolean algebra};
that is, a relatively complemented distributive lattice with bottom element.
A Boolean algebra with a top element will be called a {\em unital Boolean algebra}.
A homomorphism $\theta \colon A \rightarrow B$ of Boolean algebras is said to be {\em proper}
if $B$ is equal to the order ideal generated by the image of $\theta$.

Let $X$ be a poset.
A subset $F$ of $X$ is called {\em down directed} provided that $a,b\in F$ 
imply that there is $c\in F$ such that $c\leq a,b$, and it is called {\em upwardly closed} provided that $a\in F$ and $b\geq a$ imply $b\in F$.
A {\em filter } is a non-empty subset that is down directed and upwardly closed.
A {\em proper filter} is a filter $F$ which is a proper subset of $X$, that is $F\neq X$.
An {\em ultrafilter} is a maximal proper filter with respect to subset inclusion.
With each Boolean algebra $B$, we may associate its set of ultrafilters $B^{\ast}$.
A proper filter $F$ in a Boolean algebra is said to be {\em prime} if $a \vee b \in F$ implies that $a \in F$ or $b \in F$.
An {\em ideal} of a Boolean algebra is a non-empty subset $I$ such that $a \in I$ and $b \leq a$ implies that $b \in I$
and $a,b \in I$ implies that $a \vee b \in I$.
An proper ideal $I$ is said to be {\em prime} if $a \wedge b \in I$ implies that $a \in I$ or $b \in I$.
The following proposition summarizes some important results we shall need later, 
whose proofs are well-known in the unital case.

\begin{proposition}\label{prop: properties_of_Boolean_algebras} 
The following hold in a Boolean algebra.

\begin{enumerate}

\item Each non-zero element is contained in an ultrafilter.

\item The ultrafilters are precisely the prime filters.

\item The proper maximal ideals are precisely the prime ideals.

\item The complement of a prime ideal is a prime filter.

\item Let $I$ be an ideal and $F$ a filter such that $I \cap F = \emptyset$.
Then there exists an ideal $I'$ which is maximal such that $I \subseteq I'$ and $I' \cap F = \emptyset$.
In addition, such an ideal $I'$ is prime.
It follows that given an ideal $I$ and a filter $F$ such that $I \cap F = \emptyset$
there is an ultrafilter $F'$ such that $F \subseteq F'$ and $I \cap F' = \emptyset$.

\item For each pair of distinct non-zero elements there is an ultrafilter that contains one and omits the other.

\end{enumerate}
\end{proposition}

By a {\em Boolean space} we shall mean a Hausdorff space with a basis of compact-open subsets.
A continuous mapping of topological spaces is called {\em proper} if the inverse images of compact sets are compact.
The most famous result about Boolean algebras is the following \cite{D,Stone}.

\begin{theorem}[Stone Duality]\label{the: stone_duality} The category of Boolean algebras and proper homomorphisms
is dually equivalent to the category of Boolean spaces and proper continuous maps.
\end{theorem}

We shall now sketch out the proof of this theorem.
For each $a \in B$ define $M(a)$ to be the set of all ultrafilters containing $a$;
if $a \neq 0$ then this set is always non-empty because every non-zero element of a Boolean algebra is contained in
an ultrafilter.
Then the $M(a)$ form the basis of a topology for $B^{\ast}$ which makes $B^{\ast}$ into a Boolean space.
With each Boolean space $X$, we may associate the set $X^{\ast}$ of all compact-open subsets.
Under subset inclusion, $X^{\ast}$ is a Boolean algebra.  
The function $B \rightarrow B^{\ast \ast}$ given by $a \mapsto M(a)$ is an isomorphism of Boolean algebras.
The function $X \rightarrow X^{\ast \ast}$ given by $x \mapsto N(x)$, the set of all compact-open sets of $X$ containing $x$,
is a homeomorphism of topological spaces.

\subsection{Skew Boolean algebras}

Our reference for what follows is \cite{L3}.
A {\em right-hand skew Boolean algebra} is a triple $(B,\circ,\bullet)$ where both $(B,\circ)$ and $(B,\bullet)$ are bands
satisfying the following axioms:
\begin{enumerate}[({SB}1)]
\item\label{sb1}  $x \circ (x \bullet y) = x = (y \bullet x) \circ x$ and $x \bullet (x \circ y) = x = (y \circ x) \bullet x$.
\item\label{sb2}  $x \circ y \circ x = y \circ x$ and $x \bullet y \bullet x = x \bullet y$.
\item\label{sb3} $x \bullet y = y \bullet x$ if and only if $x \circ y = y \circ x$
\item\label{sb4} There is an element $0 \in B$ such that $0 \circ x = 0 = x \circ 0$. 
\item\label{sb5} $x^{\downarrow} = \{x \circ s \circ x \colon s \in B \}$ is a unital Boolean algebra.
\end{enumerate}
There are a number of important consequences of these axioms. 
First, from Section~3.1 of \cite{L3}, (SB2) combined with (SB5) implies that $(B,\circ)$ is right normal.
Second, $(B,\bullet,0)$ is a monoid with identity $0$.
Third, we have that
$$x \circ (y \bullet z) = (x  \circ y) \bullet (x \circ z)
\text{ and }
(y \bullet z) \circ x = (y \circ x) \bullet (z \circ x).$$
In addition, the minimum semilattice congruence $\gamma$ is the same for both $(B,\circ)$ and $(B,\bullet)$, 
and the factor-set $B/\gamma$ with respect to the operations induced by $\circ,\bullet$ and $0$ is a Boolean algebra.
The $\gamma$-classes are {\em flat} in the sense that every $\gamma$-class is a right zero semigroup with respect to $\circ$ 
and a left zero semigroup with respect to $\bullet$.
The relation $\gamma$, defined relative to $(B,\circ)$, is Green's relation $\mathscr{R}$.

\begin{remark}
{\em Our notation differs from the standard notation used in \cite{L2}, for example, 
because we use $\circ$ and $\bullet$ rather than $\wedge$ and $\vee$, respectively. 
This is to avoid ambiguity when we come to discuss meets and joins with respect to the natural partial order.}
\end{remark}

Let $(B,\circ,\bullet)$ be a right-hand skew Boolean algebra. 
Define on it a natural partial order by setting $x \leq y$ if and only if $x = x \circ y$ or, equivalently, $y = x \bullet y$.
For $x,y \in B$, we define the {\em relative complement} $x \backslash y$, of $y$ with respect to $x$,
as the relative complement of the element $y \circ x \leq x$ in the unital Boolean algebra $x^{\downarrow}$.

Let $B_1$ and $B_2$ be right-hand skew Boolean algebras. 
We call a map $\varphi \colon B_1\to B_2$ a {\em morphism of right-hand skew Boolean algebras} or just a {\em morphism} 
provided that it preserves the operations $\circ$, $\bullet$ and the zero. 
That is, we have $\varphi(x\circ y)=\varphi(x)\circ \varphi(y)$, $\varphi(x\bullet y)=\varphi(x)\bullet \varphi(y)$ 
for any $x,y\in B_1$ and $\varphi(0)=0$.  
Observe that any morphism $\varphi \colon B_1\to B_2$ induces a morphism $\overline{\varphi}:B_1/\gamma \to B_2/\gamma$ 
of Boolean algebras in a canonical way. 
We will say that $\overline{\varphi}$ {\em underlies} $\varphi$ and that $\varphi$ is {\em over} $\overline{\varphi}$.

A right-hand skew Boolean algebra is said to be a {\em right-hand skew Boolean $\wedge$-algebra} if
the meet of any two elements exists with respect to the natural partial order. 
Let $B_1,B_2$ be right-hand skew Boolean $\wedge$-algebras.  
A morphism $\varphi \colon B_1\to B_2$ will be called a {\em $\wedge$-morphisms} provided that 
$\varphi(x\wedge y)=\varphi(x)\wedge \varphi(y)$ for any $x,y\in B_1$. 
Right-hand skew Boolean $\wedge$-algebras and their $\wedge$-morphisms form a subcategory, although not a full one, 
of the category of skew Boolean algebras.

\subsection{Boolean sets}

Let $E$ be a meet semilattice equipped with the following additional data.
For each $e \in E$, let $X_{e}$ be a set
where we assume that if $e \neq f$ then $X_{e}$ and $X_{f}$ are disjoint.
If $e \geq f$ then a function $|_{f}^{e} \colon X_{e} \rightarrow X_{f}$ is given where $x \mapsto x|_{f}^{e}$.
We call these {\em restriction functions}.
In addition, $|_{e}^{e}$ is the identity on $X_{e}$ and if $e \geq f \geq g$ then 
$$(x|_{f}^{e})|_{g}^{f} = x|_{g}^{e}.$$
Put $X = \bigcup_{e \in E} X_{e}$ and define $p \colon X \rightarrow E$ by $p(x) = e$ if $x \in X_{e}$.
We shall say that $X = (X,p)$ is a {\em presheaf of sets over $E$}.
We will sometimes denote this presheaf by $X \stackrel{p}{\rightarrow} E$.
Observe that we do not assume that the sets $X_{e}$ are non-empty.
If they are all non-empty we denote the presheaf of sets by 
$X \stackrel{p}{\twoheadrightarrow} E$ and say that the presheaf has {\em global support.}

Let  $X \stackrel{p}{\rightarrow} E$ be a presheaf of sets.
Define a binary operation $\circ $ on $X$ as follows
$$x \circ y = y \vert^{p(y)}_{p(x) \wedge p(y)}.$$
It is routine to check that $(X,\circ)$ is a right normal band.
In the case where the presheaf has global support, we can also go in the opposite direction.
The following was proved in \cite{KL1}.

\begin{theorem}\label{the: presheaves_right_normal_bands}
The category of presheaves of sets with global support is equivalent to the category of right normal bands.
\end{theorem}

We now define two relations on presheaves of sets over semilattices and then explore some of their properties.
On $X$ define the relation $\leq$ by
\begin{equation}\label{eq:ord}
x \leq y \Leftrightarrow x = y\vert^{p(y)}_{p(x)}.
\end{equation}
This is a partial order.
Define the relation $\sim$ by 
$$x \sim y \Leftrightarrow \exists x \wedge y \mbox{ and } p(x \wedge y) = p(x) \wedge p(y)$$
and say that $x$ and $y$ are {\em compatible}.
The following was proved as Lemma~2.2 of \cite{KL1}.

\begin{lemma}\label{le: order_compatibility} Let $X \stackrel{p}{\rightarrow} E$ be a presheaf of sets.
\begin{enumerate}

\item If $x,y \leq z$ then $x \sim y$.

\item If $x \sim y$ and $p(x) \leq p(y)$ then $x \leq y$.

\item  $x \leq y$ if and only if $x = x \circ y$.

\item $x \sim y$ if and only if $x \circ y = y \circ x$.

\end{enumerate}
\end{lemma}

The following is immediate by the above lemma and useful in showing two elements are equal.

\begin{corollary}\label{cor: equality} 
In a presheaf of sets, if $x,y \leq z$ and $p(x) = p(y)$ then $x = y$.
\end{corollary}

The next result tells us that the map $p$ {\em reflects the partial order}.

\begin{lemma} Suppose that $b \leq p(x)$.
Then there exists a unique $y \leq x$ such that $p(y) = b$.
\end{lemma}
\begin{proof} Since $p$ is surjective, there exists $z \in X$ such that $p(z) = b \leq p(x)$.
Put $y = x \vert^{p(x)}_{p(z)}$.
By construction, we have that $p(y) = b$ and $y \leq x$. 
Uniqueness follows by Corollary~\ref{cor: equality}.
\end{proof}

The following result will be important to us.

\begin{lemma}\label{le: boolean_set_property} Let $B$ be a lattice and $X \stackrel{p}{\twoheadrightarrow} B$ a presheaf of sets over $B$.
Suppose  that $x\sim y$ and $\exists x\vee y$. 
Then $p(x\vee y)=p(x)\vee p(y)$.
\end{lemma}
\begin{proof} From $x,y \leq x \vee y$ and the fact that $p$ is order preserving, we have that $p(x),p(y) \leq p(x\vee y)$.
Thus $p(x)\vee p(y) \leq p(x \vee y)$.
It follows in particular that the restriction $t=(x\vee y)|^{p(x\vee y)}_{p(x)\vee p(y)}$ is defined.
We have $t,x\leq x\vee y$, so $t\sim x$. 
But $p(t)\geq p(x)$. 
So $t\geq x$ by Lemma~\ref{le: order_compatibility}. 
Similarly $t\geq y$. It follows that $t\geq x\vee y$. Hence $t=x\vee y$, and the statement follows.
\end{proof}

A {\em Boolean set} $X$ is a presheaf of sets $X \stackrel{p}{\twoheadrightarrow} B$ over a Boolean algebra $B$
which has a minimum element, 
usually denoted by $0$, 
with respect to the natural partial order and such that if $x \sim y$ then $\exists x \vee y$.
We also require that $p(x) = 0$ implies that $x = 0$.
It is worth stressing that Boolean sets have global support.

\begin{remark} {\em Observe that in a Boolean set, we in fact have that $x \sim y$ if and only if $\exists x \vee y$,
since by Lemma~\ref{le: order_compatibility}, if $\exists x \vee y$ then $x \sim y$.
Furthermore, by Lemma~\ref{le: boolean_set_property}, if $x \sim y$ then $p(x \vee y) = p(x) \vee p(y)$.}
\end{remark}

Let $X\stackrel{p}{\twoheadrightarrow} B_1$  and $Y\stackrel{q}{\twoheadrightarrow} B_2$ be Boolean sets. 
A {\em morphism of Boolean sets} consists of a map $\varphi \colon X\to Y$
and a morphism $\overline{\varphi} \colon B_1\to B_2$ of Boolean algebras such that the following conditions are satisfied:
\begin{enumerate}[(BM1)]
\item\label{BM1} The following diagram commutes:
$$
\xymatrix{
X\ar[d]_{p}\ar[r]^{\varphi}&Y\ar[d]^{q}\\ 
B_1\ar[r]_{\overline{\varphi}}&B_2
}
$$
That is, $\overline{\varphi}p=q\varphi$ holds.

\item\label{BM2} For any $a\geq b$ in $B_1$ the following diagram commutes:
$$
\xymatrix{
X_{a}\ar[d]_{\vert^a_{b}}\ar[r]^{\varphi}&Y_{\overline{\varphi}(a)}\ar[d]^{\vert^{\overline{\varphi}(a)}_{\overline{\varphi}(b)}}\\ 
X_{b}\ar[r]_{\varphi}&Y_{\overline{\varphi}(b)}
}
$$

That is, $\varphi(x|^a_b)=\varphi(x)|^{\overline{\varphi}(a)}_{\overline{\varphi}(b)}$ for any $x\in X_a$.
\end{enumerate}

Note that given $\varphi$, there is at most one map $\overline{\varphi}$ satisfying (BM1).
For this reason, we shall usually refer to $\varphi$ as the morphism rather than $(\varphi,\overline{\varphi})$.
Boolean sets and their morphisms form a category. 

\begin{lemma}\label{BM3}
Let $\varphi:X\to Y$ be a morphism of Boolean sets. We have $\varphi(x\vee y)=\varphi(x)\vee \varphi(y)$ for any compatible $x,y\in X$.
\end{lemma}

\begin{proof} By (BM1) we have $q(\varphi(x\vee y))=\overline{\varphi}(p(x))\vee \overline{\varphi}(p(y))$. Applying (BM2) we have $\varphi(x\vee y)\geq \varphi(x),\varphi(y)$. Thus $\varphi(x\vee y)\geq \varphi(x)\vee\varphi(y)$. This and $q(\varphi(x\vee y))=q(\varphi(x)\vee\varphi(y))$ imply the needed statement.
\end{proof}

A Boolean set $X\stackrel{p}{\twoheadrightarrow} B$ is called a {\em Boolean $\wedge$-set} provided that $x \wedge y$ exists for any $x,y\in X$. 
 Let $X\stackrel{p}{\twoheadrightarrow} B_1$ and $Y\stackrel{q}{\twoheadrightarrow} B_2$ be Boolean $\wedge$-sets. 
A morphism $\varphi \colon X\to Y$ will be called a {\em $\wedge$-morphism} provided that
$\varphi(x\wedge y)=\varphi(x)\wedge \varphi(y)$ for any $x,y\in X$. 
Boolean $\wedge$-sets and their $\wedge$-morphisms form a subcategory, although not a full one, of the category of Boolean sets. 

We now have all the definitions needed to state our first theorem.

\begin{theorem}\label{the: main_theorem1}\mbox{} 
\begin{enumerate}

\item The category of Boolean sets is isomorphic to the category of right-hand skew Boolean algebras.

\item The category of Boolean $\wedge$-sets is isomorphic to the category of right-hand skew Boolean $\wedge$-algebras. 

\end{enumerate} 
\end{theorem}

\subsection{Boolean right normal bands}
There is another way of interpreting Boolean sets which comes from Theorem \ref{the: presheaves_right_normal_bands}. Let $S$ be a right normal band. We call it {\em Boolean} provided that
$S/\gamma$ is a Boolean algebra and joins of compatible pairs of elements exist in $S$. A Boolean right normal band $S$ has (finite) meets if for any $a,b\in S$ their meet $a\wedge b$ exixts in $S$. Let $S,T$ be Boolean right normal bands and $\varphi: S\to T$ be a semigroup homomorpism. We call $\varphi$ a {\em Boolean morphism} provided that $\overline{\varphi}:S/\gamma\to T/\gamma$ is a proper morphism of Boolean algebras.  Boolean right normal bands and their Boolean morphisms form a category. Boolean right normal bands with meets and their meet-preserving Boolean morphisms also form a category that is a subcategory, although not full, of the category of Boolean right normal bands.
The following theorem can be easily deduced applying Theorem \ref{the: presheaves_right_normal_bands}.
\begin{theorem}\label{th:Bool_right_normal}\mbox{}
\begin{enumerate}

\item The category of Boolean sets is isomorphic to the category of Boolean right normal bands.
\item The category of Boolean $\wedge$-sets is isomorphic to the category of Boolean right normal bands with meets.
\end{enumerate}
\end{theorem}

\subsection{Etal\'{e} spaces}

An {\em \'{e}tal\'{e} space} is a triple $(E,p,X)$, where $E$ and $X$ are topological spaces and $p \colon E\to X$ is a surjective local homeomorphism.
We will call $X$ the {\em base space} and will also say that the \'{e}tale space $(E,p,X)$ is {\em over} $X$. 
If $x\in X$ then the set $E_x=p^{-1}(x)$ is called the {\em stalk} over $x$.
A subset $A\subseteq E$ is called an {\em open local section} or just an {\em open section}
provided that $A$ is open and the restriction of the map $p$ to $A$ is injective. 
{\em Our spaces will always have $X$ as a Boolean space}.

If $A\subseteq E$ is an open local section, we say that it is {\em over} $p(A)$. 
If $B$ is an open set in $X$ then by $E(B)$ we denote the set of all open local sections over $B$.
In the following lemma,   
whose proof follows from the fact that $p$ is a local homeomorphism, 
we collect some easy properties of \'{e}tal\'{e} spaces needed below.

\begin{lemma}\label{le: l2} Let $(E,p,X)$ be an \'{e}tal\'{e} space.
\begin{enumerate}
\item \label{i1} If $A\subseteq E$ is an open local section then $p(A)$ is open in $X$.
\item \label{i2} An open local section $A$ in $E$ is compact if and only if $p(A)$ is compact in $X$.
\item \label{i3} Let $A$ be a compact-open set in $X$. Then the set $E(A)$ is non-empty. 
\end{enumerate} 
\end{lemma}
\begin{proof} (1). This follows since local homeomorphisms are open maps.

(2). If $A$ is a compact-open local section then $p(A)$ is compact since continuous maps preserve compactness.
Suppose that $A$ is an open section such that $p(A)$ is compact.
Let $A = \bigcup_{i} V_{i}$ be an open cover.
Then $p(A) = \bigcup_{i} p(V_{i})$ is an open cover.
By compactness, we may write $p(A) = \bigcup_{i=1}^{n} p(V_{i})$.
Clearly $\bigcup_{i=1}^{n} V_{i} \subseteq A$.
Let $a \in A$.
Then $p(a) \in p(V_{i})$ for some $1 \leq i \leq n$. 
Because $A$ is a local section, we must have that $a \in V_{i}$ and the result follows.

(3). For each $a \in A$ choose by surjectivity an $e \in E$ such that $p(e) = a$.
Because $p$ is a local homeomorphism, we may find an open set $V_{e}$ in $E$ such that
$p$ induces a homeomorphism from $V_{e}$ to $p(V_{e})$.
But $X$ has a basis of compact-open sets and so, in particular, we may find a compact-open set $Y_{a}$ such that
$a \in Y_{a} \subseteq p(V_{e})$.
Passing to $Y_a\cap A$, if needed, we may assume that $Y_a\subseteq A$.
Let $U_{e}$ be $p^{-1}(Y_{a}) \cap V_{e}$.
Then $U_{e}$ is a compact-open subset of $E$ that contains $e$ and is mapped bijectively by $p$ to $Y_{a}$.
It follows that $U_{e}$ is a compact-open local section containing $e$.
The $Y_{a}$ form an open cover of $A$ and so by compactness, we may find a finite subcover.
Let $a_{1}, \ldots, a_{n}$ be the elements of $A$ such that the $Y_{i} = Y_{a_{i}}$ form a cover of $A$.
If these sets were disjoint then the sets $U_{i} = U_{e_{i}}$ would be disjoint and we could simply take their union
to form a compact-open local section over $A$.
Suppose they are not disjoint.
Form the sets $Z_{1} = Y_{1}$, $Z_{2} = Y_{2} \setminus Y_{1}$, $Z_{3} = Y_{3} \setminus (Y_{1} \cup Y_{2})$, and so on.
The sets $Z_{1}, \ldots, Z_{n}$ are disjoint and their union is $A$.
In addition, they are all open since  in a Hausdorff space compact subsets are closed.
We now define the sets $B_{i}$ associated with the $Z_{i}$ where $B_{i} \subseteq U_{i}$.
Put $B = \bigcup_{i=1}^{n} B_{i}$.
Then $p(B) = A$ and $B$ is a compact-open local section.   
\end{proof}

Let $(E,p,X)$ and $(F,q,Y)$ be \'{e}tal\'{e} spaces. 
A {\em relational morphism} $\varphi \colon (E,p,X)\to (F,q,Y)$ consists of two pieces of information:
a map $\varphi \colon E\to \mathsf{P}(F)$, where $\mathsf{P}(F)$ is the power set of $F$, 
and a map $\overline{\varphi} \colon X \to Y$ such that
$\varphi(x) \subseteq F_{\bar{\phi}(p(x))}$ for each $x \in E$.
We will say that $\overline{\varphi}$ {\em underlies} $\varphi$ and that $\varphi$ is {\em over} $\overline{\varphi}$. 

A {\em relational morphism} $\varphi \colon (E,p,X)\to (F,q,Y)$ is called a {\em partial map} provided that $|\varphi(x)|\leq 1$ for each $x\in E$.
We say that  a relational morphism is {\em locally injective} if for any $x,y\in E$ such that $p(x)=p(y)$ 
we have that $\varphi(x)\cap \varphi(y) \neq \varnothing$ implies that $x = y$. 
We say that a relational morphism is {\em locally surjective} if given $y\in F$ such that $q(y)=\overline{\varphi}(e)$ 
for some $e\in X$ then there is $x\in E$ such that $p(x)=e$ and $y\in \varphi(x)$. 
A relational morphism that is both locally injective and locally surjective will be called a {\em relational covering morphism}.
We say that a relational covering morphism $\varphi: (E,p,X)\to (F,q,Y)$ is {\em continuous} 
if for every open set $A$ in $F$ its inverse image $\varphi^{-1}(A)$ is an open set in $E$. 
We say that $\varphi$ is {\em proper} if the inverse images of compact sets are compact. 

We define the category   of \'{e}tal\'{e} spaces whose objects are  \'{e}tal\'{e} spaces over Boolean spaces, 
and in this paper we refer to the latter just as  \'{e}tal\'{e} spaces, and
whose morphisms are the  proper continuous relational covering morphisms. 
The following statement is easy to verify.

\begin{lemma}\label{lem:l1}
Let $\varphi \colon (E,p,X)\to (F,q,Y)$ be a proper continuous relational covering morphism. 
Then $\overline{\varphi} \colon X\to Y$ is a proper continuous map of topological spaces.
\end{lemma}

We define a category of Hausdorff \'{e}tal\'{e} spaces
whose objects are Hausdorff \'{e}tal\'{e} spaces and whose morphisms are proper continuous relational covering morphisms which are {\em partial maps}.
This category is a subcategory of the category of \'etal\'e spaces, although again not full.



We now have all the definitions needed to state our second theorem.

\begin{theorem}\label{the: main_theorem2} \mbox{}
\begin{enumerate}

\item The category of \'{e}tal\'{e} spaces over Boolean spaces is dually equivalent to the category of Boolean sets.

\item The category of Hausdorff \'{e}tal\'{e} spaces over Boolean spaces is dually equivalent to the category of Boolean sets with binary meets.

\end{enumerate}
\end{theorem}

Given the well-known correspondence between \'{e}tal\'{e} spaces and sheaves with global support, Theorem \ref{the: main_theorem2} (1) tells us that Boolean sets correspond to sheaves with global support over Boolean spaces. Actually,  invoking the notion of a sheaf over a category, Boolean sets themselves can be looked at as sheaves with global support over Boolean algebras as poset categories. Let us make this precise. Let $B$ be a Boolean algebra and let $c\in B$. We call a subset $S\subseteq B$ a {\em covering sieve} for $c$, provided that 
\begin{enumerate}
\item $a\leq c$ for each $a\in S$, 
\item $b\leq a$ and $a\in S$ imply $b\in S$,
\item  there is $k\geq 1$ and $a_1,\dots a_k\in S$ such that $c=a_1\vee \dots\vee a_k$.
\end{enumerate}
Let $J(c)$ be the collection of all covering sieves of $c$. Then assigning to each $c\in B$ the set $J(c)$ defines on $B$ a Grothendieck topology (see Chapter 3 of \cite{MM}). Boolean sets, being special kinds of presheaves of sets over Boolean algebras,  are then precisely sheaves with global support with respect to the described Grothendieck topology.

\section{Proof of Theorem~\ref{the: main_theorem1}}

We begin by showing that every Boolean set gives rise to a right-hand skew Boolean algebra.
The motivation for our construction comes from Examples~3.6(b) of \cite{L3}.
Let $X \stackrel{p}{\twoheadrightarrow} B$ be a Boolean set.
Let $e,f \in B$.
Define 
$$e \backslash f = (e \wedge f)'$$ 
where the complement is taken inside the unital Boolean algebra $e^{\downarrow}$.
The element $e \backslash f$ is the largest element satisfying the following two properties:
$e \backslash f \leq e$ and $(e \backslash f) \wedge f = 0$.
Let $x,y \in X$. 
We make the following definitions
$$x \circ y = y \vert^{p(y)}_{p(x) \wedge p(y)}
\quad
y \backslash x = y\vert^{p(y)}_{p(y) \backslash p(x)}
\quad
x \bullet y = x \vee (y \backslash x)$$
where the last is defined since the two parts of the join are compatible.

\begin{lemma}\label{le: properties}
Let $X \stackrel{p}{\twoheadrightarrow} B$ be a Boolean set.

\begin{enumerate}

\item $(a \backslash b) \backslash c = (a \backslash b) \wedge (a \backslash c)$. 

\item $(a \vee b) \backslash c = (a \backslash c) \vee (b \backslash c)$ if $\exists a \vee b$.

\item $a \backslash (b \vee c \backslash b) = (a \backslash b) \wedge (a \backslash c)$. 

\end{enumerate}

\end{lemma}
\begin{proof} (1) Observe that the righthand side is well-defined by 
Lemma~\ref{le: order_compatibility}.
In addition, $(a \backslash b) \backslash c \leq a$ and $(a \backslash b) \wedge (a \backslash c) \leq a$.
To prove that these two elements are equal we invoke Corollary~\ref{cor: equality}
using the fact that in a Boolean algebra we have 
$$(e \backslash f) \backslash g = (e \backslash f) \wedge (e \backslash g).$$

(2) Both sides are less than or equal to $a \vee b$.
To prove that these two elements are equal we invoke Corollary~\ref{cor: equality}
using the fact that in a Boolean algebra we have 
$$(e \vee f) \backslash i = (e \backslash i) \vee (f \backslash i).$$

(3) Both sides are less than or equal to $a$.
To prove that these two elements are equal we invoke Corollary~\ref{cor: equality}
using the fact that in a Boolean algebra we have 
$$e \backslash (i \vee j \backslash i) = (e \backslash i) \wedge (e \backslash j).$$
\end{proof}

\begin{lemma}
$(X,\bullet)$ is a band.
\end{lemma}
\begin{proof} We have that
$$(x \bullet y) \bullet z = x \vee (y \backslash x) \vee z \backslash (x \vee y \backslash x).$$
Thus using Lemma~\ref{le: properties}(3), we have that 
$$(x \bullet y) \bullet z = x \vee (y \backslash x) \vee (z \backslash x \wedge z \backslash y).$$
On the other hand,
$$x \bullet (y \bullet z) = x \vee (y \vee z \backslash y)\backslash x.$$
Using Lemma~\ref{le: properties}(2) and (1), we again obtain 
$$x \vee (y \backslash x) \vee (z \backslash x \wedge z \backslash y).$$
\end{proof}

\begin{lemma}\label{le: compatibility} The following are equivalent:
\begin{enumerate}

\item $x \sim y$.

\item $x \circ y = y \circ x$.

\item $x \bullet y = y \bullet x$.

\end{enumerate}
\end{lemma}
\begin{proof} (1)$\Leftrightarrow$(2). By Lemma~\ref{le: order_compatibility}.

(1)$\Rightarrow$(3). From $x \sim y$ we have by assumption that $\exists x \vee y$.
It follows that
$$x \bullet y = (x \backslash y) \vee (x \wedge y) \vee (y \backslash x) = y \bullet x,$$
as required.

(3)$\Rightarrow$(1)
Suppose that $x \bullet y = y \bullet x = z$.
Then $x,y \leq z$.
By Lemma~\ref{le: order_compatibility} we have that $x \sim y$.
\end{proof}

\begin{lemma} \mbox{}
\begin{enumerate}

\item $x = x \circ (x \bullet y)$.

\item $x = (y \bullet x) \circ x$.

\item $x = x \bullet (x \circ y)$.

\item $x = (y \circ x) \bullet x$.

\end{enumerate}
\end{lemma}
\begin{proof} (1) By definition  
$x \circ (x \bullet y) = (x \vee (y \backslash x))\vert^{p(x) \vee p(y \backslash x)}_{p(x)}$.
But $x \leq x \vee (y \backslash x)$.
Thus $x = x \circ (x \bullet y)$ by Lemma~\ref{le: order_compatibility}. 

(2) $(y \bullet x) \circ x = x \vert^{p(x)}_{p(x) \wedge p(y \bullet x)}$.
But $p(x) \wedge p(y \bullet x) = p(x)$ 
and so  $x = (y \bullet x) \circ x$.

(3) This equality follows from the fact that $(x \circ y) \backslash x  = 0$.

(4) This equality follows from the fact that $x = x \vert^{p(x)}_{p(x) \wedge p(y)} \vee x \vert^{p(x)}_{p(x) \backslash p(y \circ x)}$.

\end{proof}

We have now proved the following.

\begin{proposition}\label{prop: sets_to_skew} Let $X \stackrel{p}{\twoheadrightarrow} B$ be a Boolean set.
Then $(X,\circ,\bullet)$ is a right-hand skew Boolean algebra.
\end{proposition}

Note that the natural partial order on $(X,\circ,\bullet)$ coincides with the partial order on $X$ given by \eqref{eq:ord}.

The construction in the opposite direction is easier to prove.

\begin{proposition}\label{prop: skew_to_sets} Let $(X,\circ,\bullet)$ be a right-hand skew Boolean algebra.
Then $X$ is a Boolean set.
\end{proposition}
\begin{proof} Put $B = X/\mathscr{R}$.
Then $B$ is a Boolean algebra using in particular Section~1.5 of \cite{L3}.
Denote the elements of $B$ by $[x]_{R}$ and define $p(x) = [x]_{R}$.
Then $X \stackrel{p}{\twoheadrightarrow} B$ is a presheaf of sets by Theorem~\ref{the: presheaves_right_normal_bands}.
Suppose that $x \sim y$.
Clearly $y \circ x \leq x$.
But $x \wedge y \leq x$ and $p(x \wedge y) = p(x) \wedge p(y) = p(y \circ x)$.
It follows by Corollary~\ref{cor: equality} that $y \circ x = x \wedge y$.
Similarly $x \circ y = x \wedge y$.
Hence $x \circ y = y \circ x$.
But by axiom (SB1), $x = x \circ (x \bullet y)$.
Thus $x,y \leq x \bullet y$.
This and $p (x \bullet y) = p(x) \vee p(y)$ imply that $x \vee y$ exists and equals $x \bullet y$.
\end{proof}

Note that the order \eqref{eq:ord} is just the natural partial order on $(X,\circ,\bullet)$.

\begin{lemma}\label{le: mor1} Let $\varphi \colon X_1\to X_2$ be a morphism of right-hand skew Boolean algebras. Then
$\varphi$ is a morphism of Boolean sets. 
\end{lemma}
\begin{proof} Put $B_1=X_1/\gamma$ and $B_2=X_2/\gamma$ and let $p:X_1\to B_1$ and $q:X_2\to B_2$ be the projection maps. 
Let $\overline{\varphi}:B_1\to B_2$ be the morphism of Boolean algebras that underlies $\varphi$. It is immediate that (BM1) holds.  

Let $a\geq b$ in $B_1$ and let $x\in X_1$ be such that $p(x)=a$. Consider any $y\in B_1$ with $p(y)=b$. We have
$y\circ x=x|^a_b$. Since $\varphi$ preserves $\circ$, we have $\varphi(y\circ x)=\varphi(y)\circ\varphi(x)$.
On the other hand, we have that $q(\varphi(x))=\overline{\varphi}(a)$ and $q(\varphi(y))=\overline{\varphi}(b)$ by (BM1). Hence 
$\varphi(y)\circ\varphi(x)=\varphi(x)|^{\overline{\varphi}(a)}_{\overline{\varphi}(b)}$. This proves (BM2).
\end{proof}

\begin{lemma}\label{le: mor2} Let $\varphi \colon X_1\to X_2$ be a morphism of Boolean sets. 
Then $\varphi$ is a morphism of right-hand skew Boolean algebras.
\end{lemma}
\begin{proof} We have that $\varphi$ preserves $\circ$ since
$$
\varphi(x\circ y)
=\varphi(y|^{p(y)}_{p(x)\wedge p(y)})
=\varphi(y)|^{\overline{\varphi}(p(y))}_{\overline{\varphi}(p(x))\wedge \overline{\varphi}(p(y))}=\varphi(y)|^{q(\varphi(y))}_{q(\varphi(p(x))\wedge q({\varphi}(p(y))}
=\varphi(x)\circ\varphi(y)
$$
applying (BM1) and (BM2).

Note that $\varphi(y\setminus x)=\varphi(y)\setminus\varphi(x)$ since morphisms of skew Boolean algebras preserve relative complements 
or by a direct verification similar to the one above. 
Applying the definition of $\bullet$ and Lemma \ref{BM3} it follows that $\varphi$ preserves $\bullet$.
It is clear that $\varphi$ preserves the zero, since it preserves the order.
\end{proof}

Finally, we need to prove that the constructions in Propositions~\ref{prop: sets_to_skew} and \ref{prop: skew_to_sets} are mutually inverse.
To do this only requires the following lemma.

\begin{lemma} Let $(X,\circ,\bullet)$ be a right-hand skew Boolean algebra. 
Then  $x \bullet y = x \vee y \backslash x$.
\end{lemma}
\begin{proof} By axiom (SB1), $x = x \circ (x \bullet y)$ and so $x \leq x \bullet y$.
We show that $y \backslash x \leq x \bullet y$.
Using the fact that $\circ$ distributes over $\bullet$, mentioned after the axioms for right-hand skew Boolean algebras,
and the fact that 
$$x \circ (y \backslash x) = (y \backslash x) \circ x = 0$$
and
$y \geq y \backslash x$,
we calculate
$$(x \bullet y) \circ (y \backslash x) = (x \circ (y \backslash x)) \bullet (y \circ (y \backslash x)) = 0 \bullet (y \backslash x) = y \backslash x$$
and   
$$(y \backslash x) \circ (x \bullet y) = ((y \backslash x) \circ x) \bullet ((y \backslash x) \circ y) = 0 \bullet (y \backslash x) = y \backslash x.$$
It follows that $x, y \backslash x \leq x \bullet y$.
Thus $x \vee (y \backslash x) \leq x \bullet y$.
We now use the fact that under the congruence $\mathscr{R}$, whose natural map is denoted by $p$,
we have that $p(x \bullet y) = p(x \vee (y \backslash x))$.
It follows that $x \bullet y = x \vee (y \backslash x)$.
\end{proof}

\section{Proof of Theorem~\ref{the: main_theorem2} }

The proof is more complex than for our first theorem and so we split it up into steps.

\subsection{From an \'{e}tal\'{e} space to a Boolean set}\label{subs11}

In this section, we describe how to  construct a Boolean set from an \'etal\'e space $(E,p,X)$.  
Denote by $X^{\ast}$ the Boolean algebra of compact-open subsets of $X$
and by $E^{\ast}$ the set of all compact-open local sections of $p \colon E \rightarrow X$.
If $A$ is a compact-open local section in $E$ then $p(A)$ is a compact-open set in $X$ by Lemma~\ref{le: l2}.
It follows that $p$ induces a map $\widetilde{p} \colon E^{\ast} \to X^{\ast}$. 
Let $A,B$ be compact-open sets in $X$ such that $A\supseteq B$ and let $C\in E(A)$ be a compact-open local section.
Define
$$C|^A_B=C\cap p^{-1}(B)$$
and call it the {\em restriction} of $C$ from $A$ to $B$. 
It is clear that  $C|^A_B$ is a compact-open local section in $E(B)$. 
Thus $\widetilde{p} \colon E^{\ast}\to X^{\ast}$ is a presheaf of sets with global support.

\begin{proposition}\label{prop: p1} 
$E^{\ast}\stackrel{\widetilde{p}}{\twoheadrightarrow} X^{\ast}$ 
is a Boolean set. 
\end{proposition}
\begin{proof} It is clear by Lemma~\ref{le: l2}, 
that if $\widetilde{p}(x) = 0$ then $x$ is the empty local section.

Let $a,b\in E^{\ast}$ be such that $a\sim b$.  
We have to show that $\exists a\vee b$. 
We can assume that $a,b\neq \emptyset$.
Observe that $p(b)\setminus p(a)\in X^{\ast}$  and thus $b|^{p(b)}_{p(b)\setminus p(a)}\in E^{\ast}$.
Let $c=a\cup b|^{p(b)}_{p(b)\setminus p(a)}$. 
We have that $c$ is compact-open as a union of two such sets, 
and also the restriction of the map $p$ to $c$ is injective by the construction. 
It follows that $c\in E^{\ast}$. 
Clearly $c\geq a$. 
Since $a\sim b$ and $p(c)\geq p(b)$ we have $c\geq b$. 
Let $d\in E^{\ast}$ be such that $d\geq a,b$. 
Then $d\geq a, b|^{p(b)}_{p(b)\setminus p(a)}$. 
It follows that $d\geq a\cup b|^{p(b)}_{p(b)\setminus p(a)}=c$. 
Hence $c=a\vee b$. 
\end{proof}

We call the Boolean set $(E^{\ast},\widetilde{p},X^{\ast})$ the {\em dual} of the \'{e}tal\'e space $(E,p,X)$.

\subsection{From a Boolean set to an \'{e}tal\'e space} 

The passage in this direction will be a bit more involved.
  
\begin{lemma}
Let $X \stackrel{p}{\twoheadrightarrow} B$ be a Boolean set and $G$ a filter in $X$. 
Then if $x,y \in G$ then $x \circ y \in G$.
\end{lemma} 
\begin{proof} Let $x,y \in X$.
Since $G$ is downwards directed there exists $z \in G$ such that
$z \leq x,y$.
That is $z = z \circ x = z \circ y$.
Now observe that $z \circ (x \circ y) = (z \circ x) \circ y = z \circ y = z$.
It follows that $z \leq x \circ y$.
But $G$ is closed upwards and so $x \circ y \in G$.
\end{proof}

\begin{lemma}\label{le: image_is_filter}
Let $X \stackrel{p}{\twoheadrightarrow} B$ be a Boolean set.
If $G$ is a proper filter in $X$ then $p(G)$ is a proper filter in $B$.
\end{lemma}
\begin{proof} The function $p$ maps non-zero elements to non-zero elements.
Thus $p(G)$ does not contain zero.
It is clearly down directed.
We prove that it is upwardly closed.
Let $p(g) \leq b$.
Then in the Boolean algebra $B$ we may form the element $b\setminus p(g)$.
Let $x \in X$ such that $p(x) = b \setminus p(g)$.
Observe that if $y \leq g,x$ then $p(y) = 0$ and so $y = 0$.
It follows that $g \wedge (b \setminus p(g)) = 0$.
Hence $g \sim x$ and so $g \vee x$ exists.
By Lemma~\ref{le: boolean_set_property}, we have that $p(g \vee x) = b$.
But $g \vee x \in G$, as required.
\end{proof}

Let $X\stackrel{p}{\twoheadrightarrow} B$ be a Boolean set.
Denote by $X^{\ast}$ the set of ultrafilters of $X$ and by $B^{\ast}$ the set of ultrafilters of $B$. 
Let $F \subseteq B$ be an ultrafilter and $a,b\in p^{-1}(F)$. 
We say that $a$ and $b$ are {\em  conjugate over $F$},
denoted by $a \sim_{F} b$,
if there is $c \in p^{-1}(F)$ such that $c \leq a,b$.
Using Lemma~\ref{le: order_compatibility}, it is easy to show that conjugacy over $F$ is an
equivalence relation on $p^{-1}(F)$.
We denote the equivalence class containing the element $a$ by $[a]_{F}$. 

\begin{lemma}\label{lem:l2} Let $F$ be an ultrafilter in $B$.
Then $p^{-1}(F)$ is a disjoint union of ultrafilters in $X$.
Each such ultrafilter is of the form $[a]_{F}$ for some $a$ such that $p(a) \in F$.
In addition, $p([a]_{F}) = F$.
\end{lemma}
\begin{proof} 
It is immediate that $[a]_{F}$ is a filter.
We show first that $p([a]_{F})=F$. 
To do this, 
it is enough to verify the inclusion $F\subseteq p([a]_{F})$ since the opposite inclusion holds by the definition of $\sim_F$.
Let $b\in X$ is such that $p(b)\in F$. 
Let $e=a|^{p(a)}_{p(a)\wedge p(b)}$ and $f=b|^{p(b)}_{p(b)\setminus p(a)}$. 
It is clear that $e\sim f$. 
Put $c=e\vee f$. 
We have $p(c)=p(b)$ and $c \sim_{F} a$. 
It follows that $p(b)\in p([a]_{F})$.

Let $G$ be a filter in $X$ and $[a]_{F}\subseteq G$. 
Then $F=p([a]_{F})\subseteq p(G)$. 
But $p(G)$ is a filter of $B$ by Lemma~\ref{le: image_is_filter}. 
By maximality of $F$ it follows that $p(G)=F$. 
Let $b\in G$. Since $a,b\in G$ and $G$ is down directed then there is $c\in G$ such that $c\leq a,b$. 
Since also $p(c),p(b)\in F$ then $b\in [a]_{F}$. 
Hence $G\subseteq [a]_{F}$, and so $[a]_{F}= G$. 
\end{proof}

We now prove that every ultrafilter has the above form.

\begin{lemma}\label{lem:l3} 
Let $G$ be an ultrafilter of $X$. 
Then there is an ultrafilter $F$ of $B$ such that $G = [a]_{F}$ for any $a\in G$. 
In particular, $G\subseteq p^{-1}(F)$.
\end{lemma}
\begin{proof} 
By Lemma~\ref{le: image_is_filter}, we have that $p(G)$ is a filter in $B$. 
Since every filter of a Boolean algebra is contained in some ultrafilter then there is an ultrafilter $F$ in $B$ such that $p(G)\subseteq F$. 
Consider an arbitrary $a\in G$. 
We show that $G=[a]_{F}$. 
In view of the maximality of $G$ it is enough to verify only the inclusion $G\subseteq [a]_{F}$. Let $b\in G$. 
Since $G$ is down directed then there is $c\in G$ such that $c\leq a,b$. But $p(a),p(b),p(c)\in F$. 
It follows that $b\in [a]_{F}$, so that the inclusion $G\subseteq [a]_{F}$ is established.
\end{proof}

We summarize what we have found in the following.

\begin{proposition}\label{prop: ultrafilters_in_Boolean_sets}
Let $X\stackrel{p}{\twoheadrightarrow} B$ be a Boolean set.
Every ultrafilter in $X$ is of the form $[a]_{F}$ where $F$ is an ultrafilter in $B$
and $a$ is any element such that $p(a) \in F$.
We have that $[a]_{F} \cap [b]_{F} \neq \emptyset$ implies that $[a]_{F} = [b]_{F}$.
\end{proposition}

The following shows that ultrafilters in Boolean sets are `prime'.

\begin{lemma}\label{le: ultrafilters_are_prime}
Let $X\stackrel{p}{\twoheadrightarrow} B$ be a Boolean set 
and let $G$ be an ultrafilter in $X$.
If $x \sim y$ and $x \vee y \in G$ then either $x \in G$ or $y \in G$.
\end{lemma}
\begin{proof} By our results above, we may write $G = [x \vee y]_{F}$ where $F$ is an ultrafilter in $B$.
Since $x \vee y \in G$ we have that $p(x) \vee p(y) \in F$ by Lemma~\ref{le: boolean_set_property}.
By Proposition~\ref{prop: properties_of_Boolean_algebras} we have that $p(x) \in F$ or $p(y) \in F$.
Suppose that $p(x) \in F$.
It is immediate that $x \in G = [x \vee y]_{F}$, as required.
\end{proof}

It follows from Lemma \ref{lem:l2} and Lemma \ref{lem:l3} 
that the assignment $[a]_{F} \mapsto F$ defines a surjective map $\widetilde{p} \colon X^{\ast}\to B^{\ast}$. 
If $G\in X^{\ast}$ and $\widetilde{p}(G)=F$ 
then we will say 
that $G$ is {\em over} $F$ and $F$ {\em underlies} $G$. 

We now topologize the set $X^{\ast}$. 
Define 
$$L(a)=\{F\in X^{\ast}\colon a\in F\},$$
where $a$ runs through $X$. 

\begin{lemma}\label{le: properties_of_topology} 
Let $X\stackrel{p}{\twoheadrightarrow} B$ be a Boolean set.
\begin{enumerate}

\item The sets $L(a)$ form a base for a topology on $X$.

\item If $a \wedge b$ exists then $L(a \wedge b) = L(a) \cap L(b)$.

\item $L(a) = L(b)$ if and only if $a = b$.

\item $L(a) \subseteq L(b)$ if and only if $a \leq b$.

\item $L(a)$ is a local section.

\item If $ a \sim b$ then $L(a) \cup L(b) = L(a \vee b)$.

\item If $L(a) \cup L(b)$ is a local section then $a \sim b$.

\item $\widetilde{p}(L(a))=M(p( a))$ for each $a\in X$.

\item The set $L(a)$ is compact for each $ a\in X$.

\item Each compact-open local section of $X^{\ast}$ is of the form $L(a)$ for some $a \in X$.

\end{enumerate}
\end{lemma}
\begin{proof} (1). Let $a,b\in X$. 
It is enough to verify that $L(a)\cap L(b)$ can be written as a union of the sets $L(c)$, $c\in X$. 
Let $F\in L(a)\cap L(b)$. 
Then $a,b\in F$ and, since $F$ is down directed, there is $c \in F$ such that $c \leq a,b$. 
Thus $F \in L(c)$.
It is now clear that  
$$L(a)\cap L(b)= \bigcup_{0 \neq c \leq a,b} L(c).$$

(2). Straightforward.

(3). Let $L(a) = L(b)$.
Suppose first that $p(a) \neq p(b)$.
Then by Proposition~\ref{prop: properties_of_Boolean_algebras}, 
we may find an ultrafilter $F$ in $B$ such that $p(a) \in F$ and $p(b) \notin F$.
It follows that $a \in [a]_{F} = G$ but that $b \notin [a]_{F}$.
Thus $G \in L(a)$ but $G \notin L(b)$, which is a contradiction.
It follows that $p(a) = p(b)$.

Put $C = \{c \in X \colon c \leq a, b\}$.
We show that $p(C)$ is an ideal in the Boolean algebra $B$.
Because the function $p$ reflects the partial order, we have that $p(C)$ is an order ideal of $B$.
Now let $d_{1},d_{2} \in p(C)$.
Let $x_{1},x_{2} \in C$ such that $p(x_{1}) = d_{1}$ and $p(x_{2}) = d_{2}$. 
We have that $x_{1},x_{2} \leq a,b$ and so by Lemma~\ref{le: order_compatibility}, $x_{1} \sim x_{2}$.
It follows that $x_{1} \vee x_{2}$ exists and clearly $x_{1} \vee x_{2} \in C$.
Finally, $p(x_{1} \vee x_{2}) = d_{1} \vee d_{2}$ by Lemma~\ref{le: boolean_set_property} as required.

Next observe that we cannot have $p(a) = p(b) \in p(C)$ because then we would have $a = b$.
Therefore $p(a) \notin p(C)$.

By Proposition~\ref{prop: properties_of_Boolean_algebras}, 
there is an ultrafilter $F$ in the Boolean algebra $B$ such that $p(C) \cap F = \emptyset$.
Consider the ultrafilter $G = [a]_{F}$.
Suppose that $b \in G$.
Then there is some $c \leq a,b$ such that $p(c) \in F$.
But $c \in C$ and $p(c) \in p(C)$. Since $p(C)$ and $F$ are disjoint, 
this is a contradiction.
It follows that $a = b$, as required.

(4). $a\leq b$ obviously implies $L(a)\subseteq L(b)$. We now prove the reverse implication.
By (3) above, it is only necessary to prove that $L(a) = L(a \circ b)$.
Let $G\in L(a)$. Then $G\in L(b)$. So $a,b\in G$ and thus $a\circ b\in G$. It follows that
$G\in L(a\circ b)$. We have shown that $L(a) \subseteq L(a \circ b)$.
Let $G \in L(a \circ b)$.
Then $G = [a \circ b]_{p(G)}$.
But $b \in G$ and so $G = [b]_{p(G)}$.
Observe that $p(a) \wedge p(b) \in p(G)$.
Thus $p(a) \in p(G)$.
It follows that we may form the ultrafilter $[a]_{p(G)}$.
But this must contain $b$ and so $[a]_{p(G)} = [b]_{p(G)}$.
It follows that $G \in L(a)$.
We have therefore shown that $L(a) = L(a \circ b)$ and so $a = a \circ b$ giving $a \leq b$,
as required.

(5). Suppose that $G_{1},G_{2} \in L(a)$ are such that  
$\widetilde{p}(G_{1}) = \widetilde{p}(G_{2})$.
By assumption, $a \in G_{1} \cap G_{2}$.
It follows by Proposition~\ref{prop: ultrafilters_in_Boolean_sets}, that $G_{1} = G_{2}$.

(6). Clearly $L(a) \cup L(b) \subseteq L(a \vee b)$.
Let $G$ be an ultrafilter such that $a \vee b \in G$.
Then we use Lemma~\ref{le: ultrafilters_are_prime} and deduce that either $a \in G$ or $b \in G$.

(7). It is enough to prove that $L(a \circ b) = L(b \circ a)$
since then by (3), we would have $a \circ b = b \circ a$ and so from Lemma~\ref{le: order_compatibility} 
we would have $a \sim b$.
By symmetry, it is enough to prove that $L(a \circ b) \subseteq L(b \circ a)$.
Let $G \in L(a \circ b)$.
Then $b \in G$ and $p(a) \wedge p(b) \in p(G)$.
It follows that $p(a) \in p(G)$.
Thus $a \in [a]_{P(G)} = G'$ an ultrafilter.
Now $G,G' \in L(a) \cup L(b)$ and $\widetilde{p}(G) = \widetilde{p}(G')$.
Thus by assumptuion $G = G'$.
It follows that $a \in G$ and so $b \circ a \in G$.
We have shown that $G \in L(b \circ a)$, as required.

(8). Straightforward.

(9). This follows easily by (8), Lemma~\ref{lem:l3} and the compactness of the sets $M(p(a))$, $a\in X$. 

(10). Each compact-open local section can be covered by finite number of the sets $L(a)$. 
The result now follow by (6) and (7).
\end{proof}

We may now prove the following.

\begin{proposition}\label{prop:p2}  
$\widetilde{p} \colon X^{\ast} \rightarrow B^{\ast}$ is an \'{e}tal\'e space.
\end{proposition}
\begin{proof} The proof amounts to verifying that $\widetilde{p}$ is a local homeomorphism.
Let $F\in X^{\ast}$. 
Consider any $a\in X$ such $F\in L(a)$. 
We shall show that the set $L(a)$ is homeomorphic to its $\widetilde{p}$-image $M(p(a))$. 
The restriction of $\widetilde{p}$ to $L(a)$ is a bijection between $L(a)$ and $M(p(a))$
by Proposition~\ref{prop: ultrafilters_in_Boolean_sets}. 
Therefore we need only show that $\widetilde{p}$ establishes a bijection between basic opens in $L(a)$ and $M(p(a))$. 
If $A\subseteq L(a)$ is a basic compact-open set then applying Lemma~\ref{le: properties_of_topology} 
it follows that $A=L(b)$ where $b\leq a$. 
The set $\widetilde{p}(L(b))= M(p(b))$ is a basic open contained in $M(p(a))$. 
Conversely, if $A\subseteq M(p(a))$ is a basic open then $A=M(c)$ for some $c\in B$ 
such that $c\leq p(a)$. Since $a|^{p(a)}_c\leq a$ then $L(a|^{p(a)}_c)\subseteq L(a)$. We also have that $\widetilde{p}(L(a|^{p(a)}_c))=M(c)$. 
This completes the proof.
\end{proof}

We call the \'{e}tal\'e space $X^{\ast} \stackrel{\widetilde{p}}\rightarrow B^{\ast}$ the {\em dual} of the Boolean set $X\stackrel{p}{\twoheadrightarrow} B$.

\subsection{Correspondence for morphisms}

\begin{lemma} \label{lem:l5} 
Let $(E,p,X)$ and $(F,q,Y)$ be \'{e}tal\'e spaces and let $\varphi \colon (E,p,X)\to (F,q,Y)$ be a proper continuous relational covering morphism. 
Then $\varphi^{-1}$ induces a morphism, $\hat{\varphi}$, of Boolean sets from $F^{\ast}$ to $E^{\ast}$.
\end{lemma}
\begin{proof} Let $x\in F^{\ast}$ be a compact-open local section.
Since $\varphi$ is locally injective then the restriction of the map $p$ to $\varphi^{-1}(x)$ is injective
and since $\varphi$ is proper and continuous then $\varphi^{-1}(x)$ is compact-open. 
It follows that the map $\hat{\varphi}: F^{\ast}\to E^{\ast}$ given by $x\mapsto \varphi^{-1}(x)$  is well-defined.
By Lemma~\ref{lem:l1} the map $\overline{\varphi} \colon X\to Y$ is proper and continuous. 
Therefore, $\overline{\varphi}^{-1}$ induces a homomorphism of Boolean algebras from $Y^{\ast}$ to $X^{\ast}$. 
We denote it by $\overline{\hat{\varphi}}$.  

(BM1) holds. 
That is, given $x\in Y^{\ast}$ we verify that
$\widetilde{p}(\hat{\varphi}(x))=\overline{\hat{\varphi}}(\widetilde{q}(x))$.  
This equality clearly holds when $x = \emptyset$
so we may assume that $x \neq \emptyset$.
In the case $\varphi^{-1}(x)= \emptyset$,
we have $\widetilde{p}(\varphi^{-1}(x))=0$ and also $\overline{\varphi}^{-1}(\widetilde{q}(x))=0$. 
Assume that  $\varphi^{-1}(x) \neq \emptyset$. 
Then the required equality follows from $p(\varphi^{-1}(x))=\overline{\varphi}^{-1}(q(x))$ 
that holds by the construction of $\overline{\varphi}$ applying local surjectivity of $\varphi$.

(BM2) holds. 
Let $a,b\in Y^{\ast}$, $a\geq b$ and $x\in F^{\ast}_a$. 
We have to show that 
$\varphi^{-1}(x|^a_b)=\varphi^{-1}(x)|^{\overline{\varphi}^{-1}(a)}_{\overline{\varphi}^{-1}(b)}$. 
Since $x\geq x|^a_b$ then $\varphi^{-1}(x)\geq \varphi^{-1}(x|^a_b)$. By (BM1) we have that
$\widetilde{p}(\varphi^{-1}(x|^a_b))=\overline{\varphi}^{-1}(b)$
and the required equality follows.
\end{proof}

\begin{lemma}\label{lem:l6} 
Let $X\stackrel{p}{\twoheadrightarrow} B_1$ and $Y\stackrel{q}{\twoheadrightarrow} B_2$ be Boolean sets and let $\varphi: X\to Y$ 
be a morphism of Boolean sets. 
Then $\varphi^{-1}$ induces a proper continuous relational covering morphism, $\hat{\varphi}$, from $Y^{\ast}$ to $X^{\ast}$.
\end{lemma}
\begin{proof}
Let $G$ be an ultrafilter in $Y$.
By Proposition~\ref{prop: ultrafilters_in_Boolean_sets},
we can write $G = [y]_{F}$ where $y\in G$ and $F = q([y]_{F})$ is an ultrafilter in $B_2$. 
By Theorem~\ref{the: stone_duality}, 
we have that ${\overline{\varphi}}^{-1}(F)$ is an ultrafilter in $B_1$. 
By (BM1), $\varphi(x)\in [y]_{F}$ implies that $p(x)\in {\overline{\varphi}}^{-1}(F)$.
We show that 
\begin{equation}\label{eq:a}
\varphi^{-1}(G)=\bigcup_{x\in \varphi^{-1}(G)}[x]_{{\overline{\varphi}}^{-1}(F)}.
\end{equation}
To do this, we show that 
$x\in  \varphi^{-1}(G)$ implies that $[x]_{{\overline{\varphi}}^{-1}(F)}\subseteq \varphi^{-1}(G)$.
Let $t\in [x]_{{\overline{\varphi}}^{-1}(F)}$. 
Then there is $z$ such that $z\leq t,x$ and
$p(z)\in {\overline{\varphi}}^{-1}(F)$.
Since $z=x|^{p(x)}_{p(z)}$,
then 
$\varphi(z)=\varphi(x)|^{\overline{\varphi}(p(x))}_{\overline{\varphi}(p(z))}$
using (BM2).
From $\overline{\varphi}(p(x)), \overline{\varphi}(p(z))\in F$ and
$\varphi(x)\in [y]_{F}$ we conclude that $\varphi(z)\in [y]_{F}$, 
and so $z\in \varphi^{-1}([y]_{F})$. 
Now $t\geq z$ implies $\varphi(t)\geq \varphi(z)$.
It follows that $\varphi(t)\in [y]_{F}$ 
because $\varphi(z)\in [y]_{F}$ and $[y]_{\sim_F}$ is upwardly closed. 
Therefore $t\in \varphi^{-1}([y]_{F})$, and \eqref{eq:a} is established.
We can now construct $\hat{\varphi}$. Let $F\in Y^{\ast}$. 
We define $\hat{\varphi}(F)$ to be the set of all 
ultrafilters $G$ in $X$ such that $G\subseteq \varphi^{-1}(F)$. 
We put $\overline{\hat{\varphi}}=\overline{\varphi}^{-1}$. 
It follows that $\overline{\hat{\varphi}}$ is a relational morphism.

$\hat{\varphi}$ is locally injective.
Suppose that $F_{1}$ and $F_{2}$ are two ultrafilters of $Y$ such that
$q(F_{1}) = q(F_{2})$ and $\hat{\varphi}(F_{1}) \cap \hat{\varphi}(F_{2}) \neq \emptyset$.
Then $F_{1} \cap F_{2} \neq \emptyset$ and so by Proposition~\ref{prop: ultrafilters_in_Boolean_sets} 
$F_{1} = F_{2}$.

$\hat{\varphi}$ is locally surjective.
Let $H \in X^{\ast}$ be such that $p(H) = \bar{\varphi}^{-1}(F)$
where $F$ is an ultrafilter in $B_{2}$.
Let $h \in H$.
Then $\varphi (h) \in Y$.
Observe that $q(\varphi (h)) = \bar{\varphi}(p(h))$.
Thus $G = [\varphi (h)]_{F}$ is an ultrafilter in $Y$ and
$H \in \hat{\varphi} (G)$. 

Finally, to show that $\hat{\varphi}$ is proper and continuous, 
it is enough to show that $\hat{\varphi}^{-1}$ takes compact-open local sections 
to compact-open local sections. 
But this is the case, since compact-open local sections of $X^{\ast}$ are of the form $L(a)$, $a\in X$, 
and it is easy to see that $\hat{\varphi}^{-1}(L(a))=L(\varphi(a))$, 
the latter being a compact-open local section. 
\end{proof}

\begin{proposition}\mbox{}
\begin{enumerate}

\item The assignment that takes the \'etal\'e space $p \colon E \rightarrow X$
to the Boolean set $\widetilde{p} \colon E^{\ast} \rightarrow X^{\ast}$ of compact-open local sections
and a proper continuous relational covering morphism $\varphi$ to the morphism of Boolean sets $\hat{\varphi}$ is a contravariant functor.

\item The assignment that takes a Boolean set $p \colon X \rightarrow B$ to the \'etal\'e space  
$\widetilde{p} \colon X^{\ast} \rightarrow B^{\ast}$ of ultrafilters and a morphism $\varphi$ of
Boolean sets to a proper continuous relational covering morphism $\hat{\varphi}$ is a contravariant functor.

\end{enumerate}
\end{proposition}
\begin{proof} The only thing that needs verification is functoriality of these assignments. This is straightforward to show and is left to the reader. 
\end{proof}

\subsection{Proof of the duality theorem}

\begin{proposition}\label{prop11} 
Let $X = (X,p,B)$ be a Boolean set. 
Then the map $\alpha \colon X \rightarrow X^{\ast\ast}$ given by $a \mapsto L(a)$ is an isomorphism of Boolean sets.
\end{proposition}
\begin{proof} The fact that $\alpha$ is a bijection follows by Lemma~\ref{le: properties_of_topology}.
It only remains to show that $\alpha$ is an morphism of Boolean sets and $\overline{\alpha}$ is given via
$a\mapsto M(a)$, $a\in B$. Axiom (BM1) holds because $M(p(x))=\widetilde{p}(L(x))$ for each $x\in X$. 
For Axiom (BM2) let $x\in X$. Put $a=p(x)$ and let $b\leq a$. We are to check that $L(x)|^{M(a)}_{M(b)}=L(x|^a_b)$. 
But this equality holds, because $x\in F$ and $b\in p(x)$ is equivalent to saying that $x|^a_b\in F$. The proof is complete.
\end{proof}

Let $(E,f,X)$ and $(F,g,Y)$ be \'{e}tal\'e spaces. They are called {\em isomorphic}, 
provided that there exist homeomorphisms $\varphi\colon E\to F$ and $\psi\colon X\to Y$, such that $g \varphi=\psi f$.

\begin{proposition}\label{prop12} Let $E=(E,p,X)$ be an \'{e}tal\'e space. Then $E^{\ast\ast}$ is isomorphic to $E$ via the map
$\beta: a\mapsto K_a=\{x\in E^{\ast}\colon a\in x\}$, $a\in E$.
\end{proposition}
\begin{proof} We first verify that the map $\beta$ is well-defined, that is, 
that $K_a$ is an ultrafilter in $E^{\ast}$ for each $a\in E$. 
From Theorem~\ref{the: stone_duality}, we have that that $N_{p(a)}=\{y\in X^{\ast}\colon p(a)\in y\}$ is an ultrafilter in $X^{\ast}$. 
Now, applying the fact that $p$ is a local homeomorphism, it easily follows that $K_a=[x]_{{N_{p(a)}}}$ for any $x\in K_a$. 
In particular, $K_a$ is an ultrafilter of $E^{\ast}$.

$\beta$ is injective. 
Assume $a,b\in E$ and $a\neq b$. 
If $p(a)\neq p(b)$ then by Theorem~\ref{the: stone_duality} $N_{p(a)}\neq N_{p(b)}$ and so $K_a\neq K_b$ 
as $N_{p(a)}=\widetilde{p}(K_a)$ and $N_{p(b)}=\widetilde{p}(K_b)$. 
Assume now that $p(a) = p(b)$. Let $x\in K_a$. 
Since $x$ is a compact-open local section and $a\in x$ then $b\not\in x$. 
It follows that $x\not\in K_b$, and so $K_a\neq K_b$ in this case as well.

$\beta$ is surjective. 
Let $G$ be an ultrafilter in $E^*$. 
By Lemma \ref{lem:l3} and since any ultrafilter of $X^{\ast}$ is of the form $N_a$ for some $a\in X^{\ast}$ we can assume that 
$G=[x]_{{N_a}}$ for some $x\in E^{\ast}$ and $a\in X^{\ast}$ with $a\in \widetilde{p}(x)$. 
Let $y=x(a)$. 
Then $G=[x]_{{N_{p(y)}}}$ 
and therefore $G=K_{y}$ as is shown in the first paragraph of this proof.

Finally, it follows from Theorem~\ref{the: stone_duality} 
that the map $\overline{\beta}: X\to X^{\ast\ast}$ given by $a\mapsto N_a$ is a homeomorphism. 
It is straightforward to verify that both $\beta$ and $\beta^{-1}$ are continuous and that $\beta$ and $\overline{\beta}$ commute with projection maps. 
It follows that $E$ and $E^{\ast\ast}$ are isomorphic.
\end{proof}

The first part of Theorem~\ref{the: main_theorem2} now follows from Propositions \ref{prop11} and \ref{prop12}.
The second part of Theorem~\ref{the: main_theorem2} is a consequence of the first part of Theorem \ref{the: main_theorem2} 
and the following two statements.

\begin{proposition}\label{prop13} 
A Boolean set $X$ has binary meets if and only if its dual  \'{e}tal\'e space $X^{\ast}$ is Hausdorff. 
\end{proposition}
\begin{proof} 
Assume $X\stackrel{p}{\twoheadrightarrow} B$ has binary meets and let $x,y\in X^{\ast}$. 
Since the base space $B^{\ast}$ is Hausdorff it is clear that if $\widetilde{p}(x)\neq \widetilde{p}(y)$  
there are some neighborhoods of $x$ and $y$ separating them. So we can assume that $\widetilde{p}(x)= \widetilde{p}(y)$. 
Consider some $a,b\in X^{\ast\ast}$ such that $x\in a$ and $y\in b$. Restricting $a$ and $b$ to $\widetilde{p}(a)\wedge \widetilde{p}(b)$, 
if needed, we can assume that $\widetilde{p}(a)= \widetilde{p}(b)$. Since the Boolean set $X^{\ast\ast}$ has binary meets, 
we can consider $a\wedge b\in X^{\ast\ast}$. Then both $a\setminus (a\wedge b)$ and $b\setminus (a\wedge b)$ are in $X^{\ast\ast}$ 
and are disjoint neighborhoods of $x$ and $y$, respectively.

Assume  $X^{\ast}$ is Hausdorff. Let $a,b\in X^{\ast\ast}$. It is enough to show that the set-theoretic intersection $a\cap b$ also belongs to $X^{\ast\ast}$. 
By the construction of a dual Boolean set, both $a$ and $b$ are compact-open local sections. As $X^{\ast}$ is Hausdorff, then $a$ and $b$  are also closed. 
Then  $a\cap b$ is also a compact clopen local section. 
By Lemma~\ref{le: properties_of_topology} any compact-open local section in $X^{\ast}$ equals $L(c)$ 
for some $c\in X$. 
\end{proof}

\begin{proposition}\label{prop14} 
Let $X\stackrel{p}{\twoheadrightarrow} B_1$ and $Y\stackrel{q}{\twoheadrightarrow} B_2$ be Boolean sets with binary meets 
and  $\varphi \colon X\to Y$ be a morphism. 
Then $\varphi$ preserves binary meets if and only if the relational covering morphism $\hat{\varphi} \colon Y^{\ast}\to X^{\ast}$ is a partial map.
\end{proposition}
\begin{proof}
For one direction, assume that $\hat{\varphi}$ is not a partial map and show that $\varphi$ does not preserve all binary meets. 
Let $a\in Y^{\ast}$ be such that $|\hat{\varphi}(a)|\geq 2$. Let $x,y\in \hat{\varphi}(a)$. 
Since $X^{\ast}$ is a Hausforff space then there are disjoint basic neighborhoods of $x$ and $y$. 
By Lemma~\ref{le: properties_of_topology},
we can assume that $x\in L(a)$, $y\in L(b)$ and $L(a)\cap L(b)=\emptyset$. 
Then $a\wedge b=0$ and hence $\varphi(a\wedge b)=0$. 
But $\varphi(a)\wedge \varphi(b)\neq 0$ as both sections $\hat{\varphi}^{-1}(L(a))$ and  $\hat{\varphi}^{-1}(L(b))$ 
go through $a$, and so, since we are in an \'{e}tal\'e space, there is $c\in Y^{\ast\ast}$ such that $\hat{\varphi}^{-1}(L(a)),\hat{\varphi}^{-1}(L(b))\geq c$.

For the other direction, observe that morphisms of Boolean sets preserve partial order, and so
$\varphi(a\wedge b)\leq \varphi(a)\wedge \varphi(b)$ for any $a,b\in X$. 
Assume that for some $a,b$ this inequality is strict. 
This means that there is $x\in Y^\ast$ such that 
$x\in \hat{\varphi}^{-1}(L(a))\wedge \hat{\varphi}^{-1}(L(b))$ and $x\not\in \hat{\varphi}^{-1}(L(a\wedge b))$. 
Then $\hat{\varphi}(x)$ has a non-empty intersection with each of $L(a)$ and $L(b)$, and has an empty intersection with $L(a)\cap L(b)$.  
It follows that $|\hat{\varphi}(x)|\geq 2$.
\end{proof}


\end{document}